\documentclass[10pt,english]{smfart}

\usepackage[T1]{fontenc}
\usepackage[english,francais]{babel}
\usepackage{latexsym,amscd}
\usepackage{amsmath,amsfonts,amssymb,mathrsfs}
\usepackage{enumerate,euscript}
\usepackage{amssymb,url,xspace,smfthm}
\input xy
\xyoption{all}

\newcommand{\BibTeX}{{\scshape Bib}\kern-.08em\TeX}
\newcommand{\T}{\S\kern .15em\relax }
\newcommand{\AMS}{$\mathcal{A}$\kern-.1667em\lower.5ex\hbox
        {$\mathcal{M}$}\kern-.125em$\mathcal{S}$}

\DeclareMathOperator{\rang}{rk}

\DeclareMathOperator{\Spec}{Spec}

\title{Positive degree and arithmetic bigness}
\date{\today}
\author{Huayi Chen}

\address{CMLS, \'Ecole Polytechnique\\
91120, Palaiseau, France. Universit\'e Paris 8}
\email{huayi.chen@polytechnique.org}
\urladdr{http://www.math.polytechnique.fr/~chen}

\begin{document}
\def\smfbyname{}

\begin{abstract}
We establish, for a generically big Hermitian line
bundle, the convergence of truncated Harder-Narasimhan
polygons and the uniform continuity of the limit. As
applications, we prove a conjecture of Moriwaki asserting
that the arithmetic volume function is actually a limit
instead of a sup-limit, and we show how to compute the
asymptotic polygon of a Hermitian line bundle, by using
the arithmetic volume function.
\end{abstract}

\maketitle

\tableofcontents
\section{Introduction}

Let $K$ be a number field and $\mathcal O_K$ be its
integer ring. Let $\mathcal X$ be a projective arithmetic
variety of total dimension $d$ over $\Spec\mathcal O_K$.
For any Hermitian line bundle $\overline{\mathcal L}$ on
$\mathcal X$, the {\it arithmetic volume} of
$\overline{\mathcal L}$ introduced by Moriwaki (see
\cite{Moriwaki07}) is
\begin{equation}\label{Equ:volume arithm}\widehat{\mathrm{vol}}(\overline
{\mathcal L})=\limsup_{n\rightarrow\infty}
\frac{\widehat{h}^0(\mathcal X,\overline{\mathcal
L}^{\otimes n})}{n^d/d!},\end{equation}where
$\widehat{h}^0(\mathcal X,\overline{\mathcal L}^{\otimes
n}):=\log\#\{s\in H^0(\mathcal X,\mathcal L^{\otimes
n})\;|\; \forall\sigma:K\rightarrow\mathbb
C,\,\|s\|_{\sigma,\sup}\le 1\}$. The Hermitian line
bundle $\overline{\mathcal L}$ is said to be {\it
arithmetically big} if
$\widehat{\mathrm{vol}}(\overline{\mathcal L})>0$. The
notion of arithmetic bigness had been firstly introduced
by Moriwaki \cite{Moriwaki00} \S2 in a different form.
Recently he himself (\cite{Moriwaki07} Theorem 4.5) and
Yuan (\cite{Yuan07} Corollary 2.4) have proved that the
arithmetic bigness in \cite{Moriwaki00} is actually
equivalent to the strict positivity of the arithmetic
volume function \eqref{Equ:volume arithm}. In
\cite{Moriwaki07}, Moriwaki has proved the continuity of
\eqref{Equ:volume arithm} with respect to
$\overline{\mathcal L}$ and then deduced some comparisons
to arithmetic intersection numbers ({\it loc. cit.}
Theorem 6.2).

Note that the volume function \eqref{Equ:volume arithm}
is an arithmetic analogue of the classical volume
function for a line bundle on a projective variety: if
$L$ is a line bundle on a projective variety $X$ of
dimension $d$ defined over a field $k$, the {\it volume}
of $L$ is
\begin{equation}\label{Equ:volume geom}\mathrm{vol}(L):=
\displaystyle\limsup_{n\rightarrow\infty}
\frac{\rang_kH^0(X,L^{\otimes n})}{n^d/d!}.\end{equation}
Similarly, $L$ is said to be {\it big} if
$\mathrm{vol}(L)>0$. After Fujita's approximation theorem
(see \cite{Fujita94}, and \cite{Takagi07} for positive
characteristic case), the sup-limit in \eqref{Equ:volume
geom} is in fact a limit (see \cite{LazarsfeldII}
11.4.7).

During a presentation at {\it Institut de Math\'ematiques
de Jussieu}, Moriwaki has conjectured that, in arithmetic
case, the sequence $\big( \widehat{h}^0(\mathcal
X,\overline {\mathcal L})/n^d\big)_{n\ge1}$ also
converges. In other words, one
has\[\displaystyle\widehat{\mathrm{vol}}(\overline{\mathcal
L})=\lim_{n\rightarrow\infty}
\frac{\widehat{h}^0(\mathcal X,\overline{\mathcal
L}^{\otimes n})}{n^d/d!}.\] The strategy proposed by him
is to develop an analogue of Fujita's approximation
theorem in arithmetic setting (see \cite{Moriwaki07}
Remark 5.7).

In this article, we prove Moriwaki's conjecture by
establishing a convergence result of Harder-Narasimhan
polygons (Theorem \ref{Thm:convergence de polygon
general}), which is a generalization of the author's
previous work \cite{Chen08} where the main tool is the
Harder-Narasimhan filtration (indexed by $\mathbb R$) of
a Hermitian vector bundle on $\Spec\mathcal O_K$ and its
associated Borel measure. To apply the convergence of
polygons, the main idea is to compare
$\widehat{h}^0(\overline E)$, defined as the logarithm of
the number of effective points in $E$, to the {\it
positive degree} $\widehat{\deg}_+(\overline E)$, which
is the maximal value of the Harder-Narasimhan polygon of
$\overline E$. Here $\overline E$ denotes a Hermitian
vector bundle on $\Spec\mathcal O_K$. We show that the
arithmetic volume function coincides with the limit of
normalized positive degrees and therefore prove the
conjecture.

In \cite{Moriwaki07} and \cite{Yuan07}, the important
(analytical) technic used by both authors is the
estimation of the distortion function, which has already
appeared in \cite{Abbes-Bouche}. The approach in the
present work, which is similar to that in
\cite{Rumely_Lau_Varley}, relies on purely algebraic
arguments. We also establish an explicit link between the
volume function and some geometric invariants of
$\overline{\mathcal L}$ such as asymptotic slopes, which
permits us to prove that $\overline{\mathcal L}$ is big
if and only if the norm of the smallest non-zero section
of $\overline{\mathcal L}^{\otimes n}$ decreases
exponentially when $n$ tends to infinity. This result is
analogous to Theorem 4.5 of \cite{Moriwaki07} or
Corollary 2.4 (1)$\Leftrightarrow$(4) of \cite{Yuan07}
except that we avoid using analytical methods.

In our approach, the arithmetic volume function can be
interpreted as the limit of maximal values of
Harder-Narasimhan polygons. Inspired by Moriwaki's work
\cite{Moriwaki07}, we shall establish the uniform
continuity for limit of truncated Harder-Narasimhan
polygons (Theorem \ref{Thm:continuity}). This result
refines {\it loc. cit.} Theorem 5.4. Furthermore, we show
that the asymptotic polygon can be calculated from the
volume function of the Hermitian line bundle twisted by
pull-backs of Hermitian line bundles on $\Spec\mathcal
O_K$.

Our method works also in function field case. It
establishes an explicit link between the geometric volume
function and some classical geometry such as
semistability and Harder-Narasimhan filtration. This
generalizes for example a work of Wolfe \cite{Wolfe05}
(see also \cite{Ein_Las_Mus_Nak_Po05} Example 2.12)
concerning volume function on ruled varieties over
curves. Moreover, recent results in
\cite{Boucksom02,Bou_Fav_Mat06,Ber_Bou08} show that at
least in function field case, the asymptotic polygon is
``differentiable'' with respect to the line bundle, and
there may be a ``measure-valued intersection product''
from which we recover arithmetic invariants by
integration.

The rest of this article is organized as follows. We fist
recall some notation in Arakelov geometry in the second
section. In the third section, we introduce the notion of
positive degree for a Hermitian vector bundle on
$\Spec\mathcal O_K$ and we compare it to the logarithm of
the number of effective elements. The main tool is the
Riemann-Roch inequality on $\Spec\mathcal O_K$ due to
Gillet and Soul\'e \cite{Gillet-Soule91}. In the fourth
section, we establish the convergence of the measures
associated to suitably filtered section algebra of a big
line bundle (Theorem \ref{Thm:convergence de polygon
general}). We show in the fifth section that the
arithmetic bigness of $\overline{\mathcal L}$ implies the
classical one of $\mathcal L_K$, which is a
generalization of a result of Yuan \cite{Yuan07}. By the
convergence result in the fourth section, we are able to
prove that the volume of $\overline{\mathcal L}$
coincides with the limit of normalized positive degrees,
and therefore the sup-limit in \eqref{Equ:volume arithm}
is in fact a limit (Theorem \ref{Thm:volume comme une
limit}). Here we also need the comparison result in the
third section. Finally, we prove that the arithmetic
bigness is equivalent to the positivity of asymptotic
maximal slope (Theorem \ref{Thm:acritereion}). In the
sixth section, we establish the continuity of the limit
of truncated polygons. Then we show in the seventh
section how to compute the asymptotic
polygon.\vspace{2mm}

{\bf\noindent Acknowledgement } This work is inspired
by a talk of Moriwaki at the {\it Institut de
Math\'ematiques de Jussieu}. I am grateful to him for
pointing out to me that his results in
\cite{Moriwaki07} hold in continuous metric case as an
easy consequence of Weierstrass-Stone theorem. I would
like to thank J.-B. Bost for a stimulating suggestion
and helpful comments, also for having found an error in
a previous version of this article. I am also grateful
to A. Chambert-Loir, C. Mourougane and C. Soul\'e for
discussions. Most of results in the present article are
obtained and written during my visit at the {\it
Institut des Hautes \'Etudes Scientifiques}. I would
like to thank the institute for hospitalities.

\section{Notation and reminders}

Throughout this article, we fix a number field $K$ and
denote by $\mathcal O_K$ its algebraic integer ring, and
by $\Delta_K$ its discriminant. By (projective) {\it
arithmetic variety} we mean an integral projective flat
$\mathcal O_K$-scheme.
\subsection{Hermitian vector bundles}
If $\mathcal X$ is an arithmetic variety, one calls {\it
Hermitian vector bundle} on $\mathcal X$ any pair
$\overline {\mathcal E}=(\mathcal
E,(\|\cdot\|_\sigma)_{\sigma:K\rightarrow\mathbb C})$
where $\mathcal E$ is a locally free $\mathcal
O_X$-module, and for any embedding
$\sigma:K\rightarrow\mathbb C$, $\|\cdot\|_\sigma$ is a
continuous Hermitian norm on $\mathcal E_{\sigma,\mathbb
C}$. One requires in addition that the metrics
$(\|\cdot\|_\sigma)_{\sigma:K\rightarrow\mathbb C}$ are
invariant by the action of complex conjugation. The {\it
rank} of $\overline{\mathcal E}$ is just that of
$\mathcal E$. If $\rang{\mathcal E}=1$, one says that
$\overline{\mathcal E}$ is a {\it Hermitian line bundle}.
Note that $\Spec\mathcal O_K$ is itself an arithmetic
variety. A Hermitian vector bundle on $\Spec\mathcal O_K$
is just a projective $\mathcal O_K$-module equipped with
Hermitian norms which are invariant under complex
conjugation. Let $a$ be a real number. Denote by
$\overline L_a$ the Hermitian line bundle
\begin{equation}\label{Equ:la}\overline L_a:=(\mathcal
O_K,(\|\cdot\|_{\sigma,a})_{\sigma:K\rightarrow\mathbb C
}),\end{equation} where
$\|\mathbf{1}\|_{\sigma,a}=e^{-a}$, $\mathbf{1}$ being
the unit of $\mathcal O_K$.
\subsection{Arakelov degree, slope and Harder-Narasimhan polygon}
Several invariants are naturally defined for Hermitian
vector bundles on $\Spec\mathcal O_K$, notably the {\it
Arakelov degree}, which leads to other arithmetic
invariants (cf. \cite{BostBour96}). If $\overline E$ is a
Hermitian vector bundle of rank $r$ on $\Spec\mathcal
O_K$, the {\it Arakelov degree} of $\overline E$ is
defined as the real number
\[\widehat{\deg}(\overline E):=\log\#\big(E/(\mathcal O_Ks_1+
\cdots+\mathcal O_Ks_r)\big)-\frac
12\sum_{\sigma:K\rightarrow\mathbb C}\log\det\big(\langle
s_i,s_j\rangle_{\sigma}\big)_{1\le i,j\le r},\] where
$(s_i)_{1\le i\le r}$ is an element in $E^r$ which forms
a basis of $E_K$. This definition does not depend on the
choice of $(s_i)_{1\le i\le r}$. If $E$ is non-zero, the
{\it slope} of $\overline E$ is defined to be the
quotient $\widehat{\mu}(\overline
E):=\widehat{\deg}(\overline E)/\rang E$. The {\it
maximal slope} of $\overline E$ is the maximal value of
slopes of all non-zero Hermitian subbundles of $\overline
E$. The {\it minimal slope} of $\overline E$ is the
minimal value of slopes of all non-zero Hermitian
quotients of $\overline E$. We say that $\overline E$ is
{\it semistable} if $\widehat{\mu}(\overline E
)=\widehat{\mu}_{\max}(\overline E)$.

Recall that the {\it Harder-Narasimhan polygon}
$P_{\overline E}$ is by definition the concave function
defined on $[0,\rang E]$ whose graph is the convex hull
of points of the form $(\rang F,\widehat{\deg}(\overline
F))$, where $\overline F$ runs over all Hermitian
subbundles of $\overline E$. By works of Stuhler
\cite{Stuhler76} and Grayson \cite{Grayson84}, this
polygon can be determined from the Harder-Narasimhan flag
of $\overline E$, which is the only flag
\begin{equation}\label{Equ:HNflag}
E=E_0\supset E_1\supset\cdots\supset E_n=0
\end{equation}
such that the subquotients $\overline E_i/\overline
E_{i+1}$ are all semistable, and verifies
\begin{equation}\label{Equ:success slope}\widehat{\mu}(\overline E_0/\overline
E_1)<\widehat{\mu}( \overline E_1/\overline
E_2)<\cdots<\widehat{\mu}(\overline E_{n-1}/\overline E_n
).\end{equation} In fact, the vertices of $P_{\overline
E}$ are just $(\rang E_i,\widehat{\deg}(\overline E_i))$.

For details about Hermitian vector bundles on
$\Spec\mathcal O_K$, see
\cite{BostBour96,Bost2001,Chambert}.

\subsection{Reminder on Borel measures}
\label{SubSec:Reminder on Borel measures}
Denote by $C_c(\mathbb R)$ the space of all continuous
functions of compact support on $\mathbb R$. Recall that
a Borel measure on $\mathbb R$ is just a positive linear
functional on $C_c(\mathbb R)$, where the word
``positive'' means that the linear functional sends a
positive function to a positive number. One says that a
sequence $(\nu_n)_{n\ge 1}$ of Borel measures on $\mathbb
R$ converges {\it vaguely} to the Borel measure $\nu$ if,
for any $h\in C_c(\mathbb R)$, the sequence of integrals
$\big(\int h\,\mathrm{d}\nu_n\big)_{n\ge 1}$ converges to
$\int h\,\mathrm{d}\nu$. This is also equivalent to the
convergence of integrals for any $h$ in
$C_0^{\infty}(\mathbb R)$, the space of all smooth
functions of compact support on $\mathbb R$.

Let $\nu$ be a Borel probability measure on $\mathbb R$.
If $a\in\mathbb R$, we denote by $\tau_a\nu$ the Borel
measure such that $\int h\,\mathrm{d}\tau_a\nu=\int
h(x+a)\nu(\mathrm{d}x)$. If $\varepsilon>0$, let
$T_{\varepsilon}\nu$ be the Borel measure such that $\int
h\,\mathrm{d}T_{\varepsilon}\nu=\int h(\varepsilon
x)\nu(\mathrm{d}x)$.

If $\nu$ is a Borel probability measure on $\mathbb R$
whose support is bounded from above, we denote by
$P(\nu)$ the Legendre transformation (see
\cite{Homander94} II \S2.2) of the function $x\mapsto
-\int_x^{+\infty}\nu(]y,+\infty[)\,\mathrm{d}y$. It is a concave
function on $[0,1[$ which takes value $0$ at the origin.
If $\nu$ is a linear combination of Dirac measures, then
$P(\nu)$ is a {\it polygon} (that is to say, concave and
piecewise linear). An alternative definition of $P(\nu)$
is, if we denote by
$F_{\nu}^*(t)=\sup\{x\,|\,\nu(]x,+\infty[)>t\}$, then
$P(\nu)(t)=\int_0^tF_{\nu}^*(s)\,\mathrm{d}s$. One has
$P(\tau_a\nu)(t)=P(\nu)(t)+at$ and
$P(T_{\varepsilon}\nu)=\varepsilon P(\nu)$.

If $\nu_1$ and $\nu_2$ are two Borel probability measures
on $\mathbb R$, we use the symbol $\nu_1\succ\nu_2$ or
$\nu_2\prec\nu_1$ to denote the following condition:
\begin{quote}
\it for any increasing and bounded function $h$, $\int
h\,\mathrm{d}\nu_1\ge\int h\,\mathrm{d}{\nu_2}$.
\end{quote}
It defines an order on the set of all Borel probability
measures on $\mathbb R$. If in addition $\nu_1$ and
$\nu_2$ are of support bounded from above, then
$P(\nu_1)\ge P(\nu_2)$.

\subsection{Filtered spaces} Let $k$ be a field and $V$
be a vector space of finite rank over $k$. We call {\it
filtration} of $V$ any family $\mathcal F=(\mathcal
F_aV)_{a\in\mathbb R}$ of subspaces of $V$ subject to
the following conditions
\begin{enumerate}[1)]
\item for all $a,b\in\mathbb R$ such that $a\le b$, one
has $\mathcal F_aV\supset\mathcal F_bV$,
\item $\mathcal F_aV=0$ for $a$ sufficiently positive,
\item $\mathcal F_aV=V$ for $a$ sufficiently negative,
\item the function $a\mapsto\rang_k(\mathcal F_aV)$ is
left continuous.
\end{enumerate}
Such filtration corresponds to a flag
\[V=V_0\supsetneq V_1\supsetneq V_2\supsetneq \cdots\supsetneq V_n=0\]
together with a strictly increasing real sequence
$(a_i)_{0\le i\le n-1}$ describing the points where the
function $a\mapsto\rang_k(\mathcal F_aV)$ is
discontinuous.

We define a function $\lambda:V\rightarrow\mathbb
R\cup\{+\infty\}$ as follows:
\[\lambda(x)=\sup\{a\in\mathbb R\,|\,x\in\mathcal F_aV\}.\]
This function actually takes values in
$\{a_0,\cdots,a_{n-1},+\infty\}$, and is finite on
$V\setminus\{0\}$.

If $V$ is non-zero, the filtered space $(V,\mathcal F)$
defines a Borel probability measure $\nu_{V}$ which is a
linear combination of Dirac measures:
\[\nu_{V}=\sum_{i=0}^{n-1}\frac{\rang V_i-\rang V_{i+1}}{
\rang V}\delta_{a_i}.\] Note that the support of $\nu_V$
is just $\{a_0,\cdots,a_{n-1}\}$. We define
$\lambda_{\min}(V)=a_0$ and $\lambda_{\max}(V)=a_{n-1}$.
Denote by $P_V$ the polygon $P(\nu_V)$. If $V=0$, by
convention we define $\nu_V$ as the zero measure.

If $\xymatrix{0\ar[r]&V'\ar[r]&V\ar[r]&V''\ar[r]&0}$ is
an exact sequence of filtered vector spaces, where
$V\neq 0$, then the following equality holds (cf.
\cite{Chen08} Proposition 1.2.5):
\begin{equation}\label{Equ:suite exacte es mesure}\nu_V=\frac{\rang V'}{\rang V}\nu_{V'}+\frac{\rang
V''}{\rang V} \nu_{V''}.\end{equation}

If $\overline E$ is a non-zero Hermitian vector bundle on
$\Spec\mathcal O_K$, then the Harder-Narasimhan flag
\eqref{Equ:HNflag} and the successive slope
\eqref{Equ:success slope} defines a filtration of
$V=E_K$, called the {\it Harder-Narasimhan filtration}.
We denote by $\nu_{\overline E}$ the Borel measure
associated to this filtration, called the {\it measure
associated} to the Hermitian vector bundle $\overline E$.
One has the following relations:
\begin{equation}
\lambda_{\max}(V)=\widehat{\mu}_{\max}(\overline E),\quad
\lambda_{\min}(V)=\widehat{\mu}_{\min}(\overline E),\quad
P_V=P_{\overline E}=P(\nu_{\overline E}).
\end{equation}
For details about filtered spaces and their measures and
polygons, see \cite{Chen08} \S1.2.

\subsection{Slope inequality and its measure form}

For any maximal ideal $\mathfrak p$ of $\mathcal O_K$,
denote by $K_{\mathfrak p}$ the completion of $K$ with
respect to the $\mathfrak p$-adic valuation $v_{\mathfrak
p}$ on $K$, and by $|\cdot|_{\mathfrak p}$ be the
$\mathfrak p$-adic absolute value such that
$|a|_{\mathfrak p}=\#(\mathcal O_K/\mathfrak
p)^{-v_{\mathfrak p}(a)}$.

Let $\overline E$ and $\overline F$ be two Hermitian
vector bundles on $\Spec\mathcal O_K$. Let
$\varphi:E_K\rightarrow F_K$ be a non-zero $K$-linear
homomorphism. For any maximal ideal $\mathfrak p$ of
$\mathcal O_K$, let $\|\varphi\|_{\mathfrak p}$ be the
norm of the linear mapping $\varphi_{K_{\mathfrak
p}}:E_{K_\mathfrak p}\rightarrow F_{K_{\mathfrak p}}$.
Similarly, for any embedding $\sigma:K\rightarrow\mathbb
C$, let $\|\varphi\|_\sigma$ be the norm of
$\varphi_{\sigma,\mathbb C}:E_{\sigma,\mathbb
C}\rightarrow F_{\sigma,\mathbb C}$. The {\it height} of
$\varphi$ is then defined as
\begin{equation}
h(\varphi):=\sum_{\mathfrak p}\log\|\varphi\|_{\mathfrak
p}+\sum_{\sigma:K\rightarrow\mathbb
C}\|\varphi\|_{\sigma}.
\end{equation}

Recall the {\it slope inequality} as follows (cf.
\cite{BostBour96} Proposition 4.3):
\begin{prop}
If $\varphi$ is injective, then
$\widehat{\mu}_{\max}(\overline
E)\le\widehat{\mu}_{\max}(\overline F)+h(\varphi)$.
\end{prop}

The following estimation generalizing \cite{Chen08}
Corollary 2.2.6 is an application of the slope
inequality.

\begin{prop}\label{Pro:inegalite des pente}
Assume $\varphi$ is injective. Let $\theta=\rang E/\rang
F$. Then one has $\nu_{\overline F}\succ
\theta\tau_{h(\varphi)}\nu_{\overline E}
+(1-\theta)\delta_{\widehat{\mu}_{\min}(\overline F)}$.
\end{prop}
\begin{proof}
We equip $E_K$ and $F_K$ with Harder-Narasimhan
filtrations. The slope inequality implies that
$\lambda(\varphi(x))\ge\lambda(x)-h(\varphi)$ for any
$x\in E_K$ (see \cite{Chen08} Proposition 2.2.4). Let $V$
be the image of $\varphi$, equipped with induced
filtration. By \cite{Chen08} Corollary 2.2.6,
$\nu_V\succ\tau_{h(\varphi)}\nu_{\overline E}$. By
\eqref{Equ:suite exacte es mesure}, $\nu_{\overline
F}\succ\theta\nu_V+(1-\theta)\delta_{\widehat{\mu}_{\min}(\overline
F)}$, so the proposition is proved.
\end{proof}

\section{Positive degree and number of effective elements}
\label{Sec:positive degree and } Let $\overline E$ be a
Hermitian vector bundle on $\Spec\mathcal O_K$. Define
\[\widehat{h}^0(\overline E):=\log\#\{s\in E\;|\;\forall
\sigma:K\rightarrow\mathbb C,\;\|s\|_\sigma\le 1\},\]
which is the logarithm of the number of effective
elements in $E$. Note that if
$\xymatrix{0\ar[r]&\overline E'\ar[r]&\overline
E\ar[r]&\overline E''\ar[r]&0}$ is a short exact sequence
of Hermitian vector bundles, then
$\widehat{h}^0(\overline E')\le\widehat{h}^0(\overline E
)\le\widehat{h}^0(\overline E')+\widehat{h}^0(\overline
E'')$.

In this section, we define an invariant of $\overline E$,
suggested by J.-B. Bost, which is called the {\it
positive degree}:
\[\widehat{\deg}_+(\overline
E):=\displaystyle\max_{t\in[0,1]}P_{\overline E}(t).\] If
$E$ is non-zero, define the {\it positive slope} of
$\overline E$ as $\widehat{\mu}_+(\overline
E)=\widehat{\deg}_+(\overline E)/\rang E$. Using the
Riemann-Roch inequality established by Gillet and Soul\'e
\cite{Gillet-Soule91}, we shall compare
$\widehat{h}^0(\overline E)$ to
$\widehat{\deg}_+(\overline E)$.

\subsection{Reminder on dualizing bundle and Riemann-Roch inequality}
Denote by $\overline{\omega}_{\mathcal O_K}$ the {\it
arithmetic dualizing bundle} on $\Spec\mathcal O_K$: it
is a Hermitian line bundle on $\Spec\mathcal O_K$ whose
underlying $\mathcal O_K$-module is $\omega_{\mathcal O_K
}:=\mathrm{Hom}_{\mathbb Z}(\mathcal O_K,\mathbb Z)$.
This $\mathcal O_K$-module is generated by the trace map
$\mathrm{tr}_{K/\mathbb Q}:K\rightarrow\mathbb Q$ up to
torsion. We choose Hermitian metrics on $\omega_{\mathcal
O_K}$ such that $\|\mathrm{tr}_{K/\mathbb Q}\|_\sigma=1$
for any embedding $\sigma:K\rightarrow\mathbb C$. The
arithmetic degree of $\overline{\omega}_{\mathcal O_K}$
is $\log|\Delta_K|$, where $\Delta_K$ is the discriminant
of $K$ over $\mathbb Q$.

We recall below a result in \cite{Gillet-Soule91}, which
should be considered as an arithmetic analogue of
classical Riemann-Roch formula for vector bundles on a
smooth projective curve.

\begin{prop}[Gillet and Soul\'e]\label{Pro:Gillet et soule}
There exists an increasing function $C_0:\mathbb
N_*\rightarrow\mathbb R_+$ satisfying $C_0(n)\ll_K n\log
n$ such that, for any Hermitian vector bundle $\overline
E$ on $\Spec\mathcal O_K$, one has
\begin{equation}\label{EqU:RRarithmetique}
\big|\widehat{h}^0(\overline
E)-\widehat{h}^0(\overline{\omega}_{\mathcal O_K}\otimes
\overline E^\vee)-\widehat{\deg}(\overline E )\big|\le
C_0(\rang E).
\end{equation}
\end{prop}

\subsection{Comparison of $\widehat{h}^0$ and $\widehat{\deg}_+$}
Proposition \ref{Pro:comparaos} below is a comparison
between $\widehat{h}^0$ and $\widehat{\deg}_+$. The
following lemma, which is consequences of the
Riemann-Roch inequality \eqref{EqU:RRarithmetique}, is
needed for the proof.
\begin{lemm}\label{Lem:consequence de RRA}
Let $\overline E$ be a non-zero Hermitian vector bundle
on $\Spec\mathcal O_K$.
\begin{enumerate}[1)]
\item If $\widehat{\mu}_{\max}(\overline E)<0$, then $\widehat h^0(
\overline E)=0$.
\item If $\widehat{\mu}_{\min}(\overline
E)>\log|\Delta_K|$, then $\big|\widehat{h}^0(\overline E
)-\widehat{\deg}(\overline E)\big|\le C_0(\rang E)$.
\item If $\widehat{\mu}_{\min}(\overline E)\ge 0$, then
$ \big|\widehat h^0(\overline E)-\widehat{\deg}(\overline
E)\big|\le\log|\Delta_K|\rang E+C_0(\rang E) $.
\end{enumerate}
\end{lemm}
\begin{proof}
1) Assume that $\overline E$ has an effective section.
There then exists a homomorphism $\varphi:\overline
L_0\rightarrow\overline E$ whose height is negative or
zero. By slope inequality, we obtain
$\widehat{\mu}_{\max}(\overline E)\ge 0$.

2) Since $\widehat{\mu}_{\min}(\overline
E)>\log|\Delta_K|$, we have
$\widehat{\mu}_{\max}(\overline\omega_{\mathcal
O_K}\otimes\overline E^\vee)<0$. By 1),
$\widehat{h}^0(\overline\omega_{\mathcal
O_K}\otimes\overline E^\vee)=0$. Thus the desired
inequality results from \eqref{EqU:RRarithmetique}.

3) Let $a=\log|\Delta_K|+\varepsilon$ with
$\varepsilon>0$. Then $\widehat{\mu}_{\min}(\overline
E\otimes\overline L_a )>\log|\Delta_K|$. By 2),
$\widehat{h}^0(\overline E\otimes\overline
L_a)\le\widehat{\deg}(\overline E\otimes\overline
L_a)+C_0(\rang E)=\widehat{\deg}(\overline E)+a\rang
E+C_0(\rang E)$. Since $a>0$, $\widehat{h}^0(\overline E
)\le\widehat{h}^0(\overline E\otimes\overline L_a)$. So
we obtain $\widehat{h}^0(\overline
E)-\widehat{\deg}(\overline E)\le a\rang E+C_0(\rang E)$.
Moreover, \eqref{EqU:RRarithmetique} implies
$\widehat{h}^0(\overline E)-\widehat{\deg}(\overline
E)\ge\widehat{h}^0(\overline\omega_{\mathcal
O_K}\otimes\overline E^\vee)-C_0(\rang E)\ge -C_0(\rang E
)$. Therefore, we always have
$\big|\widehat{h}^0(\overline E)-\widehat{\deg}(\overline
E)\big|\le a\rang E+C_0(\rang E)$. Since $\varepsilon$ is
arbitrary, we obtain the desired inequality.
\end{proof}

\begin{prop}\label{Pro:comparaos}
The following inequality holds:
\begin{equation}\label{Equ:difference de h0 et degplus}
\big|\widehat{h}^0(\overline
E)-\widehat{\deg}_+(\overline E)\big|\le\rang
E\log|\Delta_K|+C_0(\rang E).
\end{equation}
\end{prop}
\begin{proof}
Let the Harder-Narasimhan flag of $\overline E$ be as in
\eqref{Equ:HNflag}. For any integer $i$ such that $0\le
i\le n-1$, let $\alpha_i=\widehat{\mu}(\overline
E_i/\overline E_{i+1})$. Let $j$ be the first index in
$\{0,\cdots,n-1\}$ such that $\alpha_j\ge 0$; if such
index does not exist, let $j=n$. By definition,
$\widehat{\deg}_+(\overline E)=\widehat{\deg}(\overline
E_j )$. Note that, if $j>0$, then
$\widehat{\mu}_{\max}(\overline E/\overline
E_j)=\alpha_{j-1}<0$. Therefore we always have
$h^0(\overline E/\overline E_j)=0$ and hence
$\widehat{h}^0(\overline E)=\widehat{h}^0(\overline
E_j)$.

If $j=n$, then $\widehat{h}^0(\overline
E_j)=0=\widehat{\deg}_+(\overline E)$. Otherwise
$\widehat{\mu}_{\min}(\overline E_j)=\alpha_j\ge 0$ and
by Lemma \ref{Lem:consequence de RRA} 3), we obtain
\[\big|\widehat{h}^0(\overline E_j)-\widehat{\deg}(\overline E_j)\big|
\le\rang E_j\log|\Delta_K|+C_0(\rang E_j)\le \rang
E\log|\Delta_K|+C_0(\rang E).\]
\end{proof}

\section{Asymptotic polygon of a big line bundle}
\label{Sec:asymptotoci} Let $k$ be a field and
$B=\bigoplus_{n\ge 0}B_n$ be an integral graded
$k$-algebra such that, for $n$ sufficiently positive,
$B_n$ is non-zero and has finite rank. Let $f:\mathbb
N^*\rightarrow\mathbb R_+$ be a mapping such that
$\displaystyle\lim_{n\rightarrow\infty}f(n)/n=0$. Assume
that each vector space $B_n$ is equipped with an $\mathbb
R$-filtration $\mathcal F^{(n)}$ such that $B$ is {\it
$f$-quasi-filtered} (cf. \cite{Chen08} \S3.2.1). In other
words, we assume that there exists $n_0\in\mathbb N^*$
such that, for any integer $r\ge 2$ and all homogeneous
elements $x_1,\cdots,x_r$ in $B$ respectively of degree
$n_1,\cdots,n_r$ in $\mathbb N_{\ge n_0}$, one has
\[ \lambda(x_1\cdots
x_r)\ge\sum_{i=1}^r\Big(\lambda(x_i)-f(n_i)\Big).\] We
suppose in addition that $\sup_{n\ge
1}\lambda_{\max}(B_n)/n<+\infty$. Recall below some
results in \cite{Chen08} (Proposition 3.2.4 and Theorem
3.4.3).
\begin{prop}\label{Pro:convergence des mesures}
\begin{enumerate}[1)]
\item The sequence $(\lambda_{\max}(B_n)/n)_{n\ge 1}$
converges in $\mathbb R$.
\item If $B$ is finitely generated, then the
sequence of measures $(T_{\frac 1n}\nu_{B_n})_{n\ge 1}$
converges vaguely to a Borel probability measure on
$\mathbb R$.
\end{enumerate}
\end{prop}
In this section, we shall generalize the second assertion
of Proposition \ref{Pro:convergence des mesures} to the
case where the algebra $B$ is given by global sections of
tensor power of a big line bundle on a projective
variety.

\subsection{Convergence of measures} Let $X$ be an integral projective scheme of
dimension $d$ defined over $k$ and $L$ be a {\it big}
invertible $\mathcal O_X$-module: recall that an
invertible $\mathcal O_X$-module $L$ is said to be {big}
if its {\it volume}, defined as
\[\mathrm{vol}(L):=\limsup_{n\rightarrow\infty}\frac{\rang_kH^0(X,L^{\otimes n})
}{n^d/d!},\] is strictly positive.

\begin{theo}\label{Thm:convergence de polygon general}
With the above notation, if $B=\bigoplus_{n\ge
0}H^0(X,L^{\otimes n})$, then the sequence of measures
$(T_{\frac 1n}\nu_{B_n})_{n\ge 1}$ converges vaguely to a
probability measure on $\mathbb R$.
\end{theo}
\begin{proof} For any integer $n\ge 1$, denote by $\nu_n$
the measure $T_{\frac 1n}\nu_{B_n}$. Since $L$ is big,
for sufficiently positive $n$, $H^0(X,L^{\otimes
n})\neq 0$, and hence $\nu_n$ is a probability measure.
In addition, Proposition \ref{Pro:convergence des
mesures} 1) implies that the supports of the measures
$\nu_n$ are uniformly bounded from above. After
Fujita's approximation theorem (cf.
\cite{Fujita94,Takagi07}, see also \cite{LazarsfeldII}
Ch. 11), the volume function $\mathrm{vol}(L)$ is in
fact a limit:
\[\mathrm{vol}(L)=\lim_{n\rightarrow\infty}\frac{\rang_kH^0(
X,L^{\otimes n})}{n^d/d!}.\] Furthermore, for any real
number $\varepsilon$, $0<\varepsilon<1$, there exists an
integer $p\ge 1$ together with a finitely generated
sub-$k$-algebra $A^\varepsilon$ of
$B^{(p)}=\bigoplus_{n\ge 0}B_{np}$ which is generated by
elements in $B_p$ and such that
\[\lim_{n\rightarrow\infty}
\frac{\rang_kH^0(X,L^{\otimes np})-\rang
A_n^\varepsilon}{\rang_kH^0(X,L^{\otimes
np})}\le\varepsilon.\] The graded $k$-algebra
$A^{\varepsilon}$, equipped with induced filtrations, is
$f$-quasi-filtered. Therefore Proposition
\ref{Pro:convergence des mesures} 2) is valid for
$A^\varepsilon$. In other words, If we denote by
$\nu_n^\varepsilon$ the Borel measure $T_{\frac
1{np}}\nu_{A^{\varepsilon}_n}$, then the sequence of
measures $(\nu_n^\varepsilon)_{n\ge 1}$ converges vaguely
to a Borel probability measure $\nu^\varepsilon$ on
$\mathbb R$. In particular, for any function $h\in
C_c(\mathbb R)$, the sequence of integrals $\big(\int
h\,\mathrm{d}\nu_n^\varepsilon\big)_{n\ge 1}$ is a Cauchy
sequence. This assertion is also true when $h$ is a
continuous function on $\mathbb R$ whose support is
bounded from below: the supports of the measures
$\nu_n^\varepsilon$ are uniformly bounded from above. The
exact sequence
$\xymatrix{0\ar[r]&A_n^\varepsilon\ar[r]&B_{np}\ar[r]&
B_{np}/A_n^\varepsilon\ar[r]&0}$ implies that
\[\nu_{B_{np}}=\frac{\rang A_n^\varepsilon}{\rang B_{np}}
\nu_{A_n^\varepsilon}+\frac{\rang B_{np}-\rang
A_{n}^\varepsilon}{\rang
B_{np}}\nu_{B_{np}/A_n^\varepsilon}.\] Therefore, for any
bounded Borel function $h$, one has
\begin{equation}\label{Equ:diffence of nu et nuvaerpslo}\Big|\int h\,\mathrm{d}\nu_{np}-\frac{\rang
A_n^\varepsilon}{\rang B_{np}} \int
h\,\mathrm{d}\nu_n^\varepsilon\Big|\le\|h\|_{\sup}
\frac{\rang B_{np}-\rang A_n^\varepsilon}{\rang
B_{np}}.\end{equation} Hence, for any bounded continuous
function $h$ satisfying $\inf(\mathrm{supp}(h))>-\infty$,
there exists $N_{h,\varepsilon}\in\mathbb N$ such that,
for any $n,m\ge N_{h,\varepsilon}$,
\begin{equation}\label{Equ:Cauchy sequence}\Big|\int h\,\mathrm{d}\nu_{np}-\int
h\,\mathrm{d}\nu_{mp}\Big| \le
2\varepsilon\|h\|_{\sup}+\varepsilon.\end{equation}

Let $h$ be a smooth function on $\mathbb R$ whose support
is compact. We choose two increasing continuous functions
$h_1$ and $h_2$ such that $h=h_1-h_2$ and that the
supports of them are bounded from below. Let
$n_0\in\mathbb N^*$ suffciently large such that, for any
$r\in\{n_0p+1,\cdots,n_0p+p-1\}$, one has
$H^0(X,L^{\otimes r})\neq 0$. We choose, for such $r$, a
non-zero element $e_r\in H^0(X,L^{\otimes r})$. For any
$n\in\mathbb N$ and any $r\in\{n_0p+1,\cdots,
n_0p+p-1\}$, let $M_{n,r}=e_rB_{np}\subset B_{np+r}$,
$M_{n,r}'=e_{2n_0p+p-r}M_{n,r}\subset B_{(n+2n_0+1)p}$
and denote by $\nu_{n,r}=T_{\frac{1}{np}}\nu_{M_{n,r}}$,
$\nu_{n,r}'=T_{\frac{1}{np}}\nu_{M_{n,r}'}$, where
$M_{n,r}$ and $M_{n,r}'$ are equipped with the induced
filtrations. As the algebra $B$ is $f$-quasi-filtered, we
obtain, by \cite{Chen08} Lemma 1.2.6,
$\nu_{n,r}'\succ\tau_{a_{n,r}}\nu_{n,r}\succ\tau_{b_{n,r}}\nu_{np}$,
where
\[a_{n,r}=\frac{\lambda(e_{2n_0p+p-r})-f(np+r)-f(2n_0p+p-r)}{np},
\; b_{n,r}=a_{n,r}+\frac{\lambda(e_r)-f(np)-f(r)}{np}.\]
This implies \begin{equation}\label{Equ:encadrement}\int
h_i\,\mathrm{d}\nu_{n,r}'\ge\int h_i\,\mathrm{d}
\tau_{a_{n,r}}\nu_{n,r}\ge\int
h_i\,\mathrm{d}\tau_{b_{n,r}}\nu_{np},\quad
i=1,2.\end{equation} In particular,
\begin{equation}\label{Equ:consequenece de
l'encadrement}\Big|\int
h_i\,\mathrm{d}\tau_{a_{n,r}}\nu_{n,r}-\int
h_i\,\mathrm{d}\tau_{b_{n,r}}\nu_{np}\Big|\le \Big| \int
h_i\,\mathrm{d}\nu_{n,r}'-\int
h_i\,\mathrm{d}\tau_{b_{n,r}}\nu_{np}\Big|\end{equation}
As $\displaystyle\lim_{n\rightarrow\infty}\frac{\rang
B_{(n+2n_0+1)p}-\rang B_{np}}{\rang B_{(n+2n_0+1)p}}=0$,
$\displaystyle\lim_{n\rightarrow\infty}\Big|\int
h_i\,\mathrm{d}\nu_{n,r}'-\int
h_i\,\mathrm{d}\nu_{(n+2n_0+1)p}\Big|=0$. Moreover,
$\displaystyle\lim_{n\rightarrow\infty}b_{n,r}=0$. By
\cite{Chen08} Lemma 1.2.10, we obtain
\[\displaystyle\lim_{n\rightarrow\infty}\Big|\int
h_i\,\mathrm{d}\tau_{b_{n,r}}\nu_{np}-\int
h_i\,\mathrm{d}\nu_{np}\Big|=0.\] Therefore,
\[\begin{split}&\quad\;\limsup_{n\rightarrow\infty}\Big|\int h_i\,\mathrm{d}
\nu_{n,r}'-\int
h_i\,\mathrm{d}\tau_{b_{n,r}}\nu_{np}\Big|
\\&=\limsup_{n\rightarrow\infty}\Big|\int
h_i\,\mathrm{d}\nu_{(n+2n_0+1)p}-\int
h_i\,\mathrm{d}\nu_{np}\Big|\le
2\varepsilon\|h_i\|_{\sup}+\varepsilon.\end{split}\] By
\eqref{Equ:consequenece de l'encadrement},
$\displaystyle\limsup_{n\rightarrow\infty}\Big|\int
h_i\,\mathrm{d}\tau_{a_{n,r}}\nu_{n,r}-\int
h_i\,\mathrm{d}\tau_{b_{n,r}}\nu_{np}\Big|\le2\varepsilon\|h_i\|_{\sup}
+\varepsilon$. Note that
\[\lim_{n\rightarrow\infty}\frac{\rang B_{np+r}-\rang B_{np}}
{\rang B_{np+r}}=\lim_{n\rightarrow\infty}a_{n,r}=0.\] So
\[\lim_{n\rightarrow\infty}\Big|\int h_i\,\mathrm{d}\nu_{n,r}
-\int
h_i\,\mathrm{d}\nu_{np+r}\Big|=\lim_{n\rightarrow\infty}
\Big|\int h_i\,\mathrm{d}\nu_{n,r}-\int
h_i\,\mathrm{d}\tau_{a_{n,r}}\nu_{n,r}\Big|=0.\] Hence
\[\limsup_{n\rightarrow\infty}\Big|\int h\,\mathrm{d}\nu_{np+r}
-\int h\,\mathrm{d}\nu_{np}\Big|\le
2\varepsilon(\|h_1\|_{\sup}+\|h_2\|_{\sup})+2\varepsilon.\]
According to \eqref{Equ:Cauchy sequence}, we obtain that
there exists $N_{h,\varepsilon}'\in\mathbb N^*$ such
that, for all integers $l$ and $l'$ such that $l\ge
N_{h,\varepsilon}'$, $l'\ge N_{h,\varepsilon}'$, one has
\[\Big|\int h\,\mathrm{d}\nu_{l}-\int h\,\mathrm{d}
\nu_{l'}\Big|\le
4\varepsilon(\|h_1\|_{\sup}+\|h_2\|_{\sup})+2\varepsilon\|h\|_{\sup}+
6\varepsilon,\] which implies that the sequence $(\int
h\,\mathrm{d}\nu_n)_{n\ge 1}$ converges in $\mathbb R$.

Let $I:C_0^\infty(\mathbb R)\rightarrow\mathbb R$ be the
linear functional defined by
$I(h)=\displaystyle\lim_{n\rightarrow\infty}\int
h\,\mathrm{d}\nu_n$. It extends in a unique way to a
continuous linear functional on $C_c(\mathbb R)$.
Furthermore, it is positive, and so defines a Borel
measure $\nu$ on $\mathbb R$. Finally, by
\eqref{Equ:diffence of nu et nuvaerpslo}, $|\nu(\mathbb
R)-(1-\varepsilon)\nu^{\varepsilon}(\mathbb
R)|\le\varepsilon$. Since $\varepsilon$ is arbitrary,
$\nu$ is a probability measure.
\end{proof}

\subsection{Convergence of maximal values of polygons}

If $\nu$ is a Borel probability measure on $\mathbb R$
and $\alpha\in\mathbb R$, denote by $\nu^{(\alpha)}$ the
Borel probability measure on $\mathbb R$ such that, for
any $h\in C_c(\mathbb R)$,
\[\int h\,\mathrm{d}\nu^{(\alpha)}=\int h(x)1\!\!1_{[\alpha,+\infty[}
(x)\nu(dx)+h(\alpha)\nu(]-\infty,\alpha[).\] The measure
$\nu^{(\alpha)}$ is called the {\it truncation} of $\nu$
at $\alpha$. The truncation operator preserves the order
``$\succ$''.

Assume that the support of $\nu$ is bounded from above.
The truncation of $\nu$ at $\alpha$ modifies the
``polygon'' $ P(\nu)$ only on the part with derivative
$<\alpha$. More precisely, one has
\[P(\nu)=P(\nu^{(\alpha)})\text{ on }\{t\in[0,1[\,\big|\,F_\nu^*(t)\ge\alpha\}.\]
In particular, if $\alpha\le0$, then
\begin{equation}\label{Equ:max p invariant}
\max_{t\in[0,1[}P(\nu)(t)=\max_{t\in[0,1[}P(\nu^{(\alpha)})(t).\end{equation}

The following proposition shows that given a vague
convergence sequence of Borel probability measures,
almost all truncations preserve vague limit.
\begin{prop}
Let $(\nu_n)_{n\ge 1}$ be a sequence of Borel probability
measures which converges vaguely to a Borel probability
measure $\nu$. Then there exists a countable subset $Z$
of $\mathbb R$ such that, for any $\alpha\in\mathbb
R\setminus Z$, the sequence $(\nu_n^{(\alpha)})_{n\ge 1}$
converges vaguely to $\nu^{(\alpha)}$.
\end{prop}
\begin{proof}
Let $Z$ be the set of all points $x$ in $\mathbb R$ such
that $\{x\}$ has a strictly positive mass for the measure
$\nu$. It is a countable set. Then by \cite{Bourbaki65}
IV \S5 $\mathrm{n}^\circ12$ Proposition 22, for any real
number $\alpha$ outside $Z$, $\nu^{(\alpha)}_n$ converges
vaguely to $\nu^\alpha$.
\end{proof}

\begin{coro}\label{Cor:convergence of positive slope}
Under the assumption of Theorem \ref{Thm:convergence de
polygon general}, the sequence \[\displaystyle
\big(\max_{t\in[0,1]}P_{B_n}(t)/n\big)_{n\ge 1}\]
converges in $\mathbb R$.
\end{coro}
\begin{proof}
For $n\in\mathbb N^*$, denote by $\nu_n=T_{\frac
1n}\nu_{B_n}$. By Theorem \ref{Thm:convergence de polygon
general}, the sequence $(\nu_n)_{n\ge 1}$ converges
vaguely to a Borel probability measure $\nu$. Let
$\alpha<0$ be a number such that
$(\nu_n^{(\alpha)})_{n\ge 1}$ converges vaguely to
$\nu^{(\alpha)}$. Note that the supports of
$\nu_n^{\alpha}$ are uniformly bounded. So
$P(\nu_n^{(\alpha)})$ converges uniformly to
$P(\nu^{(\alpha)})$ (see \cite{Chen08} Proposition
1.2.9). By \eqref{Equ:max p invariant},
$\displaystyle\big(\max_{t\in[0,1]}P_{B_n}(t)/n\big)_{n\ge
1}$ converges to
$\displaystyle\max_{t\in[0,1]}P(\nu)(t)$.
\end{proof}

If $V$ is a finite dimensional filtered vector space over
$k$, we shall use the expression $\lambda_+(V)$ to denote
$\displaystyle\max_{t\in[0,1]}P_V(t)$. With this
notation, the assertion of Corollary \ref{Cor:convergence
of positive slope} becomes:
$\displaystyle\lim_{n\rightarrow\infty}\lambda_+(B_n)/n$
exists in $\mathbb R$.

\begin{lemm}\label{Lem:positivite de lambdaplus}
Assume that $\nu_1 $ and $\nu_2$ are two Borel
probability measures whose supports are bounded from
above. Let $\varepsilon\in ]0,1[$ and
$\nu=\varepsilon\nu_1+(1-\varepsilon)\nu_2$. Then
\begin{equation}\label{Equ:minormation de volume}\max_{t\in[0,1]}P(\nu)(t)\ge
\varepsilon\max_{t\in[0,1]}P(\nu_1)(t).\end{equation}
\end{lemm}
\begin{proof}
After truncation at $0$ we may assume that the supports
of $\nu_1$ and $\nu_2$ are contained in $[0,+\infty[$. In
this case $\nu\succ
\varepsilon\nu_1+(1-\varepsilon)\delta_0$ and hence $
P(\nu)\ge P(\varepsilon\nu_1+(1-\varepsilon)\delta_0)$.
Since
\[P(\varepsilon \nu_1+(1-\varepsilon)\delta_0)(t)=\begin{cases}
\varepsilon
P(\nu_1)(t/\varepsilon),&t\in[0,\varepsilon],\\
\varepsilon P(\nu_1)(1),&t\in[\varepsilon,1[,
\end{cases}\]
we obtain \eqref{Equ:minormation de volume}.
\end{proof}

\begin{theo}\label{Thm:critere de bigness}
Under the assumption of Theorem \ref{Thm:convergence de
polygon general}, one has
\[\lim_{n\rightarrow\infty}\lambda_+(B_n)/n>0\quad\text{if and only if}
\quad\lim_{n\rightarrow\infty}\lambda_{\max}(B_n)/n>0.\]
Furthermore, in this case, the inequality
$\displaystyle\lim_{n\rightarrow\infty}
\lambda_+(B_n)/n\le\lim_{n\rightarrow\infty}\lambda_{\max}(B_n)/n$
holds.
\end{theo}
\begin{proof}
For any filtered vector space $V$,
$\lambda_{\max}(V)>0$ if and only if $\lambda_+(V)>0$,
and in this case one always has
$\lambda_{\max}(V)\ge\lambda_+(V)$. Therefore the
second assertion is true. Furthermore, this also
implies
\[\lim_{n\rightarrow\infty}\frac 1n\lambda_+(B_n)>0
\Longrightarrow\lim_{n\rightarrow\infty}\frac
1n\lambda_{\max}(B_n)>0.\] It suffices then to prove the
converse implication. Assume that $\alpha>0$ is a real
number such that
$\displaystyle\lim_{n\rightarrow\infty}\lambda_{\max}(B_n)/n>4\alpha$.
Choose sufficiently large $n_0\in\mathbb N$ such that
$f(n)<\alpha n$ for any $n\ge n_0$ and such that there
exists a non-zero $x_0\in B_{n_0}$ satisfying
$\lambda(x_0)\ge 4\alpha n_0$. Since the algebra $B$ is
$f$-quasi-filtered, $\lambda(x_0^m) \ge4\alpha n_0
m-mf(n)\ge 3\alpha mn_0$. By Fujita's approximation
theorem, there exists an integer $p$ divisible by $n_0$
and a subalgebra $A$ of $B^{(p)}=\bigoplus_{n\ge
0}B_{np}$ generated by a finite number of elements in
$B_p$ and such that
$\displaystyle\liminf_{n\rightarrow\infty}\rang A_n/\rang
B_{np}>0$. By possible enlargement of $A$ we may assume
that $A$ contains $x_0^{p/n_0}$. By Lemma
\ref{Lem:positivite de lambdaplus},
$\displaystyle\lim_{n\rightarrow\infty}\lambda_+(A_n)/n>0$
implies
$\displaystyle\lim_{n\rightarrow\infty}\lambda_+(B_{np})/np=
\lim_{n\rightarrow\infty}\lambda_+(B_n)/n>0$. Therefore,
we reduce our problem to the case where
\begin{enumerate}[1)]
\item $B$ is an algebra of finite type generated by
$B_1$,
\item there exists $x_1\in B_1$, $x_1\neq 0$ such that $\lambda(x_1)\ge 3\alpha$
with $\alpha>0$,
\item $f(n)<\alpha n$.
\end{enumerate}
Furthermore, by Noether's normalization theorem, we may
assume that $B=k[x_1,\cdots,x_q]$ is an algebra of
polynomials, where $x_1$ coincides with the element in
condition 2). Note that
\begin{equation}
\label{Equ:minoration de lambda}\lambda(x_1^{a_1}\cdots
x_q^{a_q})\ge\sum_{i=1}^qa_i\big(
\lambda(x_i)-\alpha\big)\ge 2\alpha
a_1+\sum_{i=2}^qa_i\big(\lambda(x_i)-\alpha\big).\end{equation}
Let $\beta>0$ such that $-\beta\le\lambda(x_i)-\alpha$
for any $i\in\{2,\cdots,q\}$. We obtain from
\eqref{Equ:minoration de lambda} that
$\lambda(x_1^{a_1}\cdots x_q^{a_q})\ge\alpha a_1$ as
soon as $a_1\ge\frac\beta\alpha\sum_{i=2}^q a_i$. For
$n\in\mathbb N^*$, let
\[\begin{split}
u_n&=\#\Big\{(a_1,\cdots,a_q)\in\mathbb N^q\,\Big|\,
a_1+\cdots+a_q=n,\;
a_1\ge\frac{\beta}{\alpha}(a_2+\cdots+a_q)
\Big\}\\&=\#\Big\{ (a_1,\cdots,a_q)\in\mathbb
N^q\,\Big|\, a_1+\cdots+a_q=n,\;
a_1\ge\frac{\beta}{\alpha+\beta}n \Big\}\\
&=\binom{n-\lfloor\frac{\beta}{\alpha+\beta}n\rfloor+q-1}{q-1},
\end{split}\]and
\[v_n=\#\big\{(a_1,\cdots,a_q)\in\mathbb N^q\,\big|\,
a_1+\cdots+a_q=n \big\}=\binom{n+q-1}{q-1}.\] Thus
$\displaystyle\lim_{n\rightarrow\infty}u_n/v_n=\Big(\frac{\alpha}
{\alpha+\beta}\Big)^{q-1}>0$, which implies
$\displaystyle\lim_{n\rightarrow\infty}\frac
1n\lambda_+(B_n)>0$ by Lemma \ref{Lem:positivite de
lambdaplus}.
\end{proof}

\section{Volume function as a limit and arithmetic
bigness}\label{Sec:volume est une limite}

Let $\mathcal X$ be an arithmetic variety of dimension
$d$ and $\overline{\mathcal L}$ be a Hermitian line
bundle on $\mathcal X$. Denote by $X=\mathcal X_K$ and
$L=\mathcal L_K$. Using the convergence result
established in the previous section, we shall prove that
the volume function is in fact a limit of normalized
positive degrees. We also give a criterion of arithmetic
bigness by the positivity of asymptotic maximal slope.

\subsection{Volume function and asymptotic positive degree}
For any $n\in\mathbb N$, we choose a Hermitian vector
bundle $\pi_*(\overline{\mathcal L}^{\otimes
n})=(\pi_*(\mathcal L^{\otimes
n}),(\|\cdot\|_\sigma)_{\sigma:K\rightarrow\mathbb C})$
whose underlying $\mathcal O_K$-module is $\pi_*(\mathcal
L^{\otimes n})$ and such that
\[\max_{0\neq s\in \pi_*(\mathcal L^{\otimes n})}\Big|
\log\|s\|_{\sup}-\log\|s\|_\sigma\Big|\ll \log n,\qquad
n>1.\] Denote by $r_n$ the rank of $\pi_*(\mathcal
L^{\otimes n})$. One has $r_n\ll n^{d-1}$. For any
$n\in\mathbb N$, define
\[\widehat{h}^0(\mathcal X,\overline{\mathcal
L}^{\otimes n}):=\log\#\{s\in H^0(\mathcal X,\mathcal
L^{\otimes n})\;|\; \forall\sigma:K\rightarrow\mathbb
C,\,\|s\|_{\sigma,\sup}\le 1\}.\] Recall that the
arithmetic volume function of $\overline{\mathcal L}$
defined by Moriwaki (cf. \cite{Moriwaki07}) is
\[\widehat{\mathrm{vol}}(\overline{\mathcal L}):=
\limsup_{n\rightarrow\infty}\frac{\widehat{h}^0(\mathcal
X,\overline{\mathcal L}^{\otimes n})}{n^d/d!},\] and
$\overline{\mathcal L}$ is said to be big if and only if
$\widehat{\mathrm{vol}}(\overline{\mathcal L})>0$ (cf.
\cite{Moriwaki07} Theorem 4.5 and \cite{Yuan07} Corollary
2.4).

In the following, we give an alternative proof of a
result of Morkwaki and Yuan.
\begin{prop}
If $\overline{\mathcal L}$ is big, then $L$ is big on $X$
in usual sense.
\end{prop}
\begin{proof}
For any integer $n\ge 1$, we choose two Hermitian vector
bundles $\overline E_n^{(1)}=(\pi_*(\mathcal L^{\otimes
n}),(\|\cdot\|_\sigma^{(1)})_{\sigma:K\rightarrow\mathbb
C})$ and $\overline E_n^{(2)}=(\pi_*(\mathcal L^{\otimes
n}),(\|\cdot\|_\sigma^{(2)})_{\sigma:K\rightarrow\mathbb
C})$ such that \[\|s\|_{\sigma}^{(1)}\le
\|s\|_{\sigma,\sup} \le\|s\|_{\sigma}^{(2)}\le
r_n\|s\|_\sigma^{(1)},\] where $r_n$ is the rank of
$\pi_*(\mathcal L^{\otimes n})$. This is always possible
due to an argument of John and L\"owner ellipsoid, see
\cite{Gaudron07} definition-theorem 2.4. Thus
$\widehat{h}^0(\overline
E_n^{(2)})\le\widehat{h}^0(\mathcal X,\overline{\mathcal
L}^{\otimes n})\le\widehat{h}^0(\overline E_n^{(1)})$.
Furthermore, by \cite{Chen08} Corollay 2.2.9,
$\big|\widehat{\deg}_+(\overline
E_n^{(1)})-\widehat{\deg}_+(\overline E_n^{(2)})\big|\le
r_n\log r_n $. By \eqref{Equ:difference de h0 et
degplus}, we obtain \[\big|\widehat{h}^0(X,\overline
L^{\otimes n})-\widehat{h}^0( \overline
E_n^{(1)})\big|\le 2r_n\log|\Delta_K|+2C_0(r_n)+r_n\log
r_n.\] Furthermore, $\big|\widehat{\deg}_+(\overline
E_n^{(1)})-\widehat{\deg}_+(\pi_*(\overline{\mathcal
L}^{\otimes n }))\big|=O(r_n\log r_n)$. Hence
\[\big|\widehat{h}^0(X,\overline{\mathcal L}^{\otimes
n})-\widehat{h}^0(\pi_*(\overline{\mathcal L}^{\otimes n
}))\big|=O(r_n\log r_n).\] Since $r_n\ll n^{d-1}$, we
obtain
\begin{equation}\label{Equ:comparason de deg et h}\displaystyle\lim_{n\rightarrow\infty}
\bigg| \frac{\widehat{h}^0(\mathcal X,\overline{\mathcal
L}^{\otimes n})}{n^d/d!}-
\frac{\widehat{\deg}_+(\pi_*(\overline{\mathcal
L}^{\otimes n}) )}{n^d/d!}\bigg|=0,
\end{equation}
and therefore $\displaystyle
\widehat{\mathrm{vol}}(\overline{\mathcal L}
)=\limsup_{n\rightarrow\infty}\frac{\widehat{\deg}_+(
\pi_*(\overline{\mathcal L}^{\otimes n}) )}{n^d/d!}$. If
$\overline{\mathcal L}$ is big, then
$\widehat{\mathrm{vol}}(\overline{\mathcal L} )>0$, and
hence $\pi_*(\mathcal L^{\otimes n})\neq 0$ for $n$
sufficiently positive. Combining with the fact that
\[\displaystyle\limsup_{n\rightarrow+\infty}
\frac{\widehat{\deg}_+(\pi_*(\mathcal L^{\otimes n
}))}{nr_n}\le
\lim_{n\rightarrow+\infty}\frac{\widehat{\mu}_{\max} (
\pi_*(\overline{\mathcal L}^{\otimes n}))}{n}<+\infty,\]
we obtain $\displaystyle\limsup_{n\rightarrow+\infty}
\frac{r_n}{n^{d-1}}>0$.
\end{proof}

\begin{theo}
\label{Thm:volume comme une limit} The following
equalities hold:
\begin{equation}\label{Equ:vol comm limet}\widehat{\mathrm{vol}}(\overline{\mathcal L})=
\lim_{n\rightarrow\infty} \frac{\widehat{h}^0(\mathcal
X,\overline{\mathcal L}^{\otimes
n})}{n^{d}/d!}=\lim_{n\rightarrow\infty}
\frac{\widehat{\deg}_+( \pi_*(\overline{\mathcal
L}^{\otimes
n}))}{n^d/d!}=\mathrm{vol}(L)\lim_{n\rightarrow\infty}
\frac{\widehat{\mu}_+(\pi_*(\overline{\mathcal
L}^{\otimes n}))}{n/d},\end{equation} where the positive
slope $\widehat{\mu}_+$ was defined in
\S\ref{Sec:positive degree and }.
\end{theo}
\begin{proof}
We first consider the case where $L$ is big. The graded
algebra $B=\bigoplus_{n\ge 0}H^0(X,L^{\otimes n})$
equipped with Harder-Narasimhan filtrations is
quasi-filtered for a function of logarithmic increasing
speed at infinity (see \cite{Chen08} \S4.1.3). Therefore
Corollary \ref{Cor:convergence of positive slope} shows
that the sequence $(\lambda_+(B_n)/n)_{n\ge 1}$ converges
in $\mathbb R$. Note that
$\lambda_+(B_n)=\widehat{\mu}_+(\pi_*(\overline{\mathcal
L}^{\otimes n}))$. So the last limit in \eqref{Equ:vol
comm limet} exists. Furthermore, $L$ is big on $X$, so
\[\mathrm{vol}(L)=\lim_{n\rightarrow\infty}\frac{\rang(\pi_*(\overline{\mathcal
L}^{\otimes n}))}{n^{d-1}/(d-1)!},\] which implies the
existence of the third limit in \eqref{Equ:vol comm
limet} and the last equality. Thus the existence of the
first limit and the second equality follow from
\eqref{Equ:comparason de deg et h}.

When $L$ is not big, since
\[\lim_{n\rightarrow\infty}
\frac{\widehat{\mu}_+(\pi_*(\overline{\mathcal
L}^{\otimes n}))}{n/d}\le \lim_{n\rightarrow\infty}
\frac{\widehat{\mu}_{\max}(\pi_*(\overline{\mathcal
L}^{\otimes n}))}{n/d}<+\infty\] the last term in
\eqref{Equ:vol comm limet} vanishes. This implies the
vanishing of the second limit in \eqref{Equ:vol comm
limet}. Also by \eqref{Equ:comparason de deg et h}, we
obtain the vanishing of the first limit.

\end{proof}

\begin{coro}\label{Cor:HS moriwaki}
The following relations hold:
\begin{equation}\label{Equ:HS moriwar}\widehat{\mathrm{vol}}(\overline{\mathcal L})\ge
\limsup_{n\rightarrow\infty} \frac{\widehat{\deg}(
\pi_*(\overline{\mathcal L}^{\otimes
n}))}{n^d/d!}=\limsup_{n\rightarrow\infty}\frac{\chi(
\pi_*(\overline{\mathcal L}^{\otimes
n}))}{n^d/d!}.\end{equation}
\end{coro}
\begin{proof}
The inequality is a consequence of Theorem
\ref{Thm:volume comme une limit} and the comparison
$\widehat{\deg}_+(\overline E)\ge\widehat{\deg}(\overline
E)$. Here $\overline E$ is an arbitrary Hermitian vector
bundle on $\Spec\mathcal O_K$. The equality follows from
a classical result which compares Arakelov degree and
Euler-Poincar\'e characteristic (see \cite{Chen08} 4.1.4
for a proof). Attention: in \cite{Chen08}, the author has
adopted the convention $\widehat{\mu}(\overline
E)=\widehat{\deg}(\overline E)/[K:\mathbb Q]\rang E$.
\end{proof}

\begin{rema}
Corollary \ref{Cor:HS moriwaki} is a
generalization of \cite{Moriwaki07} Theorem 6.2 to
continuous metrics case.
\end{rema}

\subsection{A criterion of arithmetic bigness}
We shall prove that the bigness of $\overline{\mathcal
L}$ is equivalent to the positivity of the asymptotic
maximal slope of $\overline{\mathcal L}$. This result is
strongly analogous to Theorem 4.5 of \cite{Moriwaki07}.
In fact, by a result of Borek \cite{Borek05} (see also
\cite{Bost_Kunnemann} Proposition 3.3.1), which
reformulate Minkowski's First Theorem, the maximal slope
of a Hermitian vector bundle on $\Spec\mathcal O_K$ is
``close'' to the opposite of the logarithm of its first
minimum. So the positivity of the asymptotic maximal
slope is equivalent to the existence of (exponentially)
small section when $n$ goes to infinity.
\begin{theo}\label{Thm:acritereion}
$\overline{\mathcal L}$ is big if and only if
$\displaystyle\lim_{n\rightarrow\infty}
{\widehat{\mu}_{\max}( \pi_*(\overline{\mathcal
L}^{\otimes n}))}/{n}>0$. Furthermore, the following
inequality holds:
\[\displaystyle\frac{\widehat{\mathrm{vol}}(\overline{\mathcal
L})}{d\mathrm{vol}(L)}\le
\lim_{n\rightarrow\infty}\frac{\widehat{\mu}_{\max}
(\pi_*(\overline{\mathcal L}^{\otimes n}))}{n}.\]
\end{theo}
\begin{proof} Since both conditions imply the
bigness of $L$, we may assume that $L$ is big. Let
$B=\bigoplus_{n\ge 0}H^0(X,L^{\otimes n})$ equipped with
Harder-Narasimhan filtrations. One has
\[\widehat{\mu}_+(\pi_*(\overline{\mathcal L}^{\otimes n}))=
\lambda_+(B_n),\qquad\widehat{\mu}_{\max}(\pi_*(\overline{\mathcal
L}^{\otimes n})=\lambda_{\max}(B_n).\] Therefore, the
assertion follows from Theorems \ref{Thm:critere de
bigness} and \ref{Thm:volume comme une limit}.
\end{proof}

\begin{rema}\label{Rem:decroissance de 1er minimum}
After \cite{Bost_Kunnemann} Proposition 3.3.1, for any
non-zero Hermitian vector bundle $\overline E$ on
$\Spec\mathcal O_K$, one has
\begin{equation}\label{Equ:Bost Kunnemann}\Big|\widehat{\mu}_{\max}(\overline E)+
\frac12\log \inf_{0\neq s\in
E}\sum_{\sigma:K\rightarrow\mathbb C} \|s\|_{\sigma}^2
\Big| \le\frac{1}{2}\log[K:\mathbb
Q]+\frac{1}{2}\log\rang
E+\frac{\log|\Delta_K|}{2[K:\mathbb Q]}.\end{equation}
Therefore, by \eqref{Equ:Bost Kunnemann}, the bigness of
$\overline{\mathcal L}$ is equivalent to each of the
following conditions:
\begin{enumerate}[1)]
\item $L$ is big, and there
exists $\varepsilon>0$ such that, when $n$ is
sufficiently large, $\overline{\mathcal L}^{\otimes n}$
has a global section $s_n$ satisfying
$\|s_n\|_{\sigma,\sup}\le e^{-\varepsilon n}$ for any
$\sigma:K\rightarrow\mathbb C$.
\item $L$ is big, and there exists an integer $n\ge 1$
such that $\overline{\mathcal L}^{\otimes n}$ has a
global section $s_n$ satisfying $\|s_n\|_{\sigma,\sup}<1$
for any $\sigma:K\rightarrow\mathbb C$.
\end{enumerate}
Thus we recover a result of Moriwaki (\cite{Moriwaki07}
Theorem 4.5 (1)$\Longleftrightarrow$(2)).
\end{rema}

\begin{coro}\label{Cor:generaiclla big adbi}
Assume $L$ is big. Then there exists a Hermitian line
bundle $\overline M$ on $\Spec\mathcal O_K$ such that
$\overline{\mathcal L}\otimes\pi^*\overline M$ is
arithmetically big.
\end{coro}

\section{Continuity of truncated asymptotic polygon}

Let us keep the notation of \S\ref{Sec:volume est une
limite} and assume that $L$ is big on $X$. For any
integer $n\ge 1$, denote by $\nu_n$ the dilated measure
$T_{\frac 1n}\nu_{\pi_*(\overline{\mathcal L}^{\otimes
n})}$. Recall that in \S\ref{Sec:asymptotoci} we have
actually established the followint result:

\begin{prop}
\begin{enumerate}[1)]
\item the sequence of Borel measures $(\nu_n)_{n\ge 1}$
converges vaguely to a Borel probability measure $\nu$;
\item there exists a countable subset $Z$ of $\mathbb R$
such that, for any $\alpha\in\mathbb R\setminus Z$, the
sequence of polygons $(P(\nu_n^{(\alpha)}))_{n\ge 1}$
converges uniformly to $P(\nu^{(\alpha)})$, which impies
in particular that $P(\nu^{(\alpha)})$ is Lipschitz.
\end{enumerate}
\end{prop}
Let $Z$ be as in the proposition above. For any
$\alpha\in\mathbb R\setminus Z$, denote by
$P_{\overline{\mathcal L}}^{(\alpha)}$ the concave
function $P(\nu^{(\alpha)})$ on $[0,1]$. The following
property of $P_{\overline{\mathcal L}}^{(\alpha)}$
results from the definition:
\begin{prop}
For any integer $p\ge 1$, on has $P_{\overline{\mathcal
L}^{\otimes p}}^{(p\alpha)}=pP_{\overline{\mathcal
L}}^{(\alpha)}$.
\end{prop}
\begin{proof}
By definition $T_{\frac 1n}\nu_{\pi_*(\overline{\mathcal
L}^{\otimes pn})}=T_{p}\nu_n$. Using
$(T_p\nu_n)^{(p\alpha)}=T_{p}\nu_n^{(\alpha)}$, we obtain
the desired equality.
\end{proof}

\begin{rema}
We deduce from the previous proposition the equality
$\widehat{\mathrm{vol}}(\overline{\mathcal L}^{\otimes
p})=p^d\widehat{\mathrm{vol}}(\overline{\mathcal L})$,
which has been proved by Moriwaki
(\cite{Moriwaki07} Proposition 4.7).
\end{rema}

The main purpose of this section is to establish the
following continuity result, which is a generalization of
the continuity of the arithmetic volume function proved
by Moriwaki (cf. \cite{Moriwaki07} Theorem 5.4).
\begin{theo}\label{Thm:continuity}
Assume $\overline{\mathscr L}$ is a Hermitian line bundle
on $\mathcal X$. Then, for all but countably many
$\alpha\in\mathbb R$, the sequence of functions
$\Big(\frac{1}{p}P^{(p\alpha)}_{\overline{\mathcal
L}^{\otimes p}\otimes\overline{\mathscr L}}\Big)_{p\ge
1}$ converges uniformly to $P_{\overline{\mathcal
L}}^{(\alpha)}$.
\end{theo}

\begin{coro}[\cite{Moriwaki07} Theorem 5.4]
With the assumption of Theorem \ref{Thm:continuity}, one
has \[\displaystyle\lim_{p\rightarrow\infty}\frac{1}{p^d}
\widehat{\mathrm{vol}}(\overline{\mathcal L}^{\otimes
p}\otimes\overline{\mathscr
L})=\widehat{\mathrm{vol}}(\overline{\mathcal L}).\]
\end{coro}

In order to prove Theorem \ref{Thm:continuity}, we need
the following lemma.
\begin{lemm}\label{Lem:contiuite1}
Let $\overline{\mathscr L}$ be an arbitrary Hermitian
line bundle on $\Spec\mathcal O_K$. If
$\overline{\mathcal L}$ is arithmetically big, then there
exists an integer $q\ge 1$ such that $\overline{\mathcal
L}^{\otimes q}\otimes\overline{\mathscr L}$ is
arithmetically big and has at least one non-zero {\rm
effective} global section, that is, a non-zero section
$s\in H^0(\mathcal X,\mathcal L^{\otimes
q}\otimes\mathscr L)$ such that $\|s\|_{\sigma,\sup}\le
1$ for any embedding $\sigma:K\rightarrow\mathbb C$.
\end{lemm}
\begin{proof}
As $\overline{\mathcal L}$ is arithmetically big, we
obtain that $L$ is big on $X$. Therefore, there exists an
integer $m_0\ge 1$ such that $L^{\otimes
m_0}\otimes\mathscr L_K$ is big on $X$ and
$\pi_*(\mathcal L^{\otimes m_0}\otimes\mathscr L)\neq 0$.
Pick an arbitrary non-zero section $s\in H^0(\mathcal
X,\mathcal L^{\otimes m_0}\otimes\mathscr L)$ and let
$M=\displaystyle\sup_{\sigma:K\rightarrow\mathbb
C}\|s\|_{\sigma,\sup}$. After Theorem
\ref{Thm:acritereion} (see also Remark
\ref{Rem:decroissance de 1er minimum}), there exists
$m_1\in\mathbb N$ such that $\mathcal L^{\otimes m_1}$
has a section $s'$ such that $\|s'\|_{\sigma,\sup}\le
(2M)^{-1}$ for any $\sigma:K\rightarrow\mathbb C$. Let
$q=m_0+m_1$. Then $s\otimes s'$ is a non-zero strictly
effective section of $\overline{\mathcal L}^{\otimes
q}\otimes\overline{\mathscr L}$. Furthermore,
$\overline{\mathcal L}^{\otimes
q}\otimes\overline{\mathscr L}$ is arithmetically big
since it is generically big and has a strictly effective
section.
\end{proof}

\begin{proof}[Proof of Theorem \ref{Thm:continuity}]
After Corollary \ref{Cor:generaiclla big adbi}, we may
assume that $\overline{\mathcal L}$ is arithmetically
big. Let $q\ge 1$ be an integer such that
$\overline{\mathcal L}^{\otimes
q}\otimes\overline{\mathscr L}$ is arithmetically big and
has a non-zero effective section $s_1$ (cf. Lemma
\ref{Lem:contiuite1}). For any integers $p$ and $n$ such
that $p>q$, $n\ge 1$, let $\varphi_{p,n}:\pi_*(\mathcal
L^{\otimes(p-q)n})\rightarrow\pi_*(\mathcal L^{\otimes
pn}\otimes\mathscr L^{\otimes n})$ be the homomorphism
defined by the multiplication by $s_1^{\otimes n}$. Since
$s_1$ is effective, $h(\varphi_{p,n})\le 0$. Let
\[\theta_{p,n}=\rang(\pi_*(\mathcal
L^{\otimes(p-q)n}))/\rang(\pi_*(\mathcal L^{\otimes
pn}\otimes\mathscr L^{\otimes n})).\] Note that
\[\displaystyle\lim_{n\rightarrow\infty}\theta_{p,n}=\mathrm{vol}(
L^{\otimes(p-q)})/\mathrm{vol}(L^{\otimes
p}\otimes\mathscr L_K).\] Denote by $\theta_p$ this
limit. Let $\nu_{p,n}$ be the measure associated to
$\pi_*(\overline{\mathcal L}^{\otimes
pn}\otimes\overline{\mathscr L}^{\otimes n})$. Let
$a_{p,n}=\widehat{\mu}_{\min}(\pi_*(\overline{\mathcal
L}^{\otimes pn}\otimes\overline{\mathscr L}^{\otimes
n}))$. After Proposition \ref{Pro:inegalite des pente},
one has
${\nu}_{p,n}\succ\theta_{p,n}T_{(p-q)n}\nu_{(p-q)n}+(1-\theta_{p,n})
\delta_{a_{p,n}}$, or equivalently
\begin{equation}\label{Equ:encadrement 1}
T_{\frac{1}{np}}\nu_{p,n}\succ\theta_{p,n}T_{(p-q)/p}\nu_{(p-q)n}
+(1-\theta_{p,n})\delta_{a_{p,n}/np}.
\end{equation} As $L^{\otimes p}\otimes\mathscr L_K$ is
big, the sequence of measures $(T_{\frac
1n}\nu_{p,n})_{n\ge 1}$ converges vaguely to a Borel
probability measure $\eta_p$. By truncation and then by
passing $n\rightarrow\infty$, we obtain from
\eqref{Equ:encadrement 1} that for all but countably many
$\alpha\in\mathbb R$, \begin{equation}\label{Equ:mior de
mesure}(T_{\frac
1p}\eta_p)^{(\alpha)}\succ\theta_p(T_{(p-q)/p}\nu)^{(\alpha)}
+(1-\theta_p)\delta_\alpha,\end{equation} where we have
used the trivial estimation
$\delta_a^{(\alpha)}\succ\delta_{\alpha}$.

Now we apply Lemma \ref{Lem:contiuite1} on the dual
Hermitian line bundle $\overline{\mathscr L}^{\vee}$ and
obtain that there exists an integer $r\ge 1$ and an
effective section $s_2$ of $\overline{\mathcal
L}^{\otimes r}\otimes\overline{\mathscr L}^{\vee}$.
Consider now the homomorphism $\psi_{p,n}:\pi_*(\mathcal
L^{\otimes pn}\otimes\mathscr L^{\otimes
n})\rightarrow\pi_*(\mathcal L^{\otimes(p+r)n})$ induced
by multiplication by $s_2^{\otimes n}$. Its height is
negative. Let
\[\vartheta_{p,n}=\rang(\pi_*(\mathcal
L^{\otimes pn}\otimes\mathscr L^{\otimes
n}))/\rang(\pi_*(\mathcal L^{\otimes(p+r)n})).\] When $n$
tends to infinity, $\vartheta_{p,n}$ converges to
\[\vartheta_p:=\mathrm{vol}(L^{\otimes p}\otimes\mathscr L_K)/
\mathrm{vol}(L^{\otimes(p+r)}).\] By a similar argument
as above, we obtain that for all but countably many
$\alpha\in\mathbb R$,
\begin{equation}\label{Equ:encadrement partie 2}
(T_{(p+r)/p}\nu)^{(\alpha)}\succ\vartheta_{p}(T_{\frac
1p}\eta_p)^{(\alpha)}+(1-\vartheta_p)\delta_\alpha.
\end{equation}
We obtain from \eqref{Equ:mior de mesure} and
\eqref{Equ:encadrement partie 2} the following
estimations of polygons
\begin{gather}\vartheta_p^{-1}P((T_{(p+r)/p}
\nu)^{(\alpha)})(\vartheta_pt)\ge P((T_{\frac{1}{p}}\eta_p)^{(\alpha)})(t)\\
P((T_{\frac{1}{p}}\eta_p)^{(\alpha)})(t)\ge\begin{cases}
\theta_p P((T_{(p-q)/p}\nu)^{(\alpha)})(t/\theta_p),&
0\le
t\le\theta_p,\\
\theta_p
P((T_{(p-q)/p}\nu)^{(\alpha)})(1)+\alpha(t-\theta_p),&\theta_p\le
t\le 1.
\end{cases}.
\end{gather}
Finally, since
$\displaystyle\lim_{p\rightarrow\infty}\theta_p=\lim_{p\rightarrow\infty}
\vartheta_p=1$ (which is a consequence of the continuity
of the geometric volume function), combined with the fact
that both $T_{(p-q)/p}\nu$ and $T_{(p+r)/p}\nu$ converge
vaguely to $\nu$ when $p\rightarrow\infty$, we obtain,
for all but countably many $\alpha\in\mathbb R$, the
uniform convergence of $P((T_{\frac
1p}\eta_p)^{(\alpha)})$ to $P(\nu^{(\alpha)})$.
\end{proof}

\section{Compuation of asymptotic polygon by volume function}

In this section we shall show how to compute the asymptotic polygon of a Hermitian line bundle by using arithmetic volume function. Our main method is the Legendre transformation of concave functions. Let us begin with a lemma concerning Borel measures.

\begin{lemm}
Let $\nu$ be a Borel measure on $\mathbb R$ whose support is bounded from below. Then
\begin{equation}
\label{Equ:volume comme une integrale}
\max_{t\in[0,1[} P(\nu)(t)=\int_{\mathbb R}x_+\nu(\mathrm{d}x),
\end{equation}
where $x_+$ stands for $\max\{x,0\}$.
\end{lemm}
\begin{proof}
Since the function $F_{\nu}^*$ defined in
\S\ref{SubSec:Reminder on Borel measures} is essentially
the inverse of the distribution function of $\nu$, by
definition we obtain that, if $\eta$ is a Borel measure
of compact support, then
\[P(\eta)(1):=\lim_{t\rightarrow 1-}P(\eta)(t)=\int_{\mathbb R} x\eta(\mathrm{d}x).\]
Applying this equality on $\eta=\nu^{(0)}$, we obtain
\[\max_{t\in[0,1[} P(\nu)(t)=P(\nu^{(0)})(1)=
\int_{\mathbb R}x\nu^{(0)}(\mathrm{d}x)=\int_{\mathbb R}
x_+\nu^{(0)}(\mathrm{d}x)=\int_{\mathbb R}
x_+\nu(\mathrm{d}x).\]
\end{proof}

Now let $\mathcal X$ be an arithmetic variety of total dimension $d$. For any Hermitian line bundle $\overline{\mathcal L}$ on $\mathcal X$ whose generic fibre is big, we denote by $\nu_{\overline{\mathcal L}}$ the vague limite of the sequence of measures $(T_{\frac 1n}\nu_{\pi_*(\overline{\mathcal L}^{\otimes n})})_{n\ge  1}$. The existence of $\nu_{\overline{\mathcal L}}$ has been established in Theorem \ref{Thm:convergence de polygon general}.

\begin{prop}\label{Pro:calcul de polygone limite} Let $L=\mathcal L_K$.
For any real number $a$, one has
\[\int_{\mathbb R}(x-a)_+\nu_{\overline{\mathcal L}}(\mathrm{d}x)=
\frac{\widehat{\mathrm{vol}}(\overline{\mathcal L}
\otimes\pi^*\overline L_{-a})}{d\mathrm{vol}(L)},\]
where $\overline L_{-a}$ is the Hermitian line bundle on $\Spec\mathcal O_K$ defined in \eqref{Equ:la}.
\end{prop}
\begin{proof}
If $\overline M$ is a Hermitian line bundle on $\Spec\mathcal O_K$, one has the equality \[\nu_{\overline {\mathcal L}\otimes\pi^*\overline M}=\tau_{\widehat{\deg}(\overline M)}\nu_{\overline L}.\]
Applying this equality on $\overline M=\overline L_{-a}$, one obtains
\[\frac{\widehat{\mathrm{vol}}(\overline{\mathcal L}
\otimes\pi^*\overline L_{-a})}{d\mathrm{vol}(L)}=
\int_{\mathbb R}x_+\tau_{-a}\nu_{\overline{\mathcal L}}(\mathrm{d}x)=\int_{\mathbb R}(x-a)_+\nu_{\overline{\mathcal L}}(\mathrm{d}x).\]
\end{proof}

\begin{rema}
Proposition \ref{Pro:calcul de polygone limite} calculates actually the polygone $P(\nu_{\overline{\mathcal L}})$. In fact, one has
\[-\int_a^{+\infty}\nu_{\overline{\mathcal L}}(]y,+\infty[)\mathrm{d}y=-\int_{\mathbb R}(s-a)_+\nu_{\overline{\mathcal L}}(\mathrm{d}s).\]
Applying the Legendre transformation, we obtain the polygone $P(\nu_{\overline{\mathcal L}})$.
\end{rema}

\backmatter
\bibliography{chen}
\bibliographystyle{smfplain}

\end{document}